\documentclass[12pt]{amsart}
\usepackage{amssymb}
\usepackage{amsfonts,amssymb,amsmath}
\usepackage{amsthm, eucal, eufrak}

\makeatletter
\def\blfootnote{\xdef\@thefnmark{}\@footnotetext}
\makeatother

\usepackage{color}

\author{Pavel Shumyatsky}
\address{Department of Mathematics, University of Bras\'ilia,
Bras\'ilia-DF, 70910-900, Brazil }
\email{pavel@unb.br}


\title{Profinite groups with pronilpotent centralizers}

\newtheorem{theorem}{Theorem}[section]
\newtheorem{lemma}[theorem]{Lemma}

\newtheorem{problem}[theorem]{Problem}

\theoremstyle{definition}

\let\leq=\leqslant
\let\geq=\geqslant

\begin{document}

\begin{abstract} The article deals with profinite groups in which the centralizers are pronilpotent (CN-groups). It is shown that such groups are virtually pronilpotent. More precisely, let $G$ be a profinite CN-group, and let $F$ be the maximal normal pronilpotent subgroup of $G$. It is shown that $F$ is open and the structure of the finite quotient $G/F$ is described in detail.
\end{abstract}
\blfootnote{This work was supported by FAPDF and CNPq-Brazil}
\maketitle

\section{Introduction} Profinite groups in which all centralizers are pronilpotent are called CN-groups. Finite CN-groups are a classical subject in the theory of finite groups due to the role that they have played in the proof of the Feit--Thompson theorem which states that finite groups of odd order are soluble \cite{fetho}. Groups in which all centralizers are abelian are called CA-groups. Suzuki proved that finite CA-groups of odd order are soluble \cite{suzu57}. Next, Feit, Hall and Thompson extended Suzuki's result to CN-groups \cite{fehatho}. Finally, Feit and Thompson proved that all finite groups of odd order are soluble.

In recent years infinite CA-groups attracted significant interest due to their deep relation with residually free groups. Namely, finitely generated residually free CA-groups are limit groups that played a key role in the solutions of Tarski problems. Kochloukova and Zalesski introduced in \cite{kz11} a pro-p analog of limit groups via the operation of extension of centralizers which are pro-$p$ CA-groups. Further examples of pro-$p$ CA-groups include pro-$p$ completions of surface groups \cite{kz10} and of many
3-manifold groups \cite{wil17} (see also \cite{zz}).

The article \cite{ppz} deals with general questions on the structure of profinite CA-groups. It is shown there that profinite CA-groups are virtually pronilpotent. More precisely, it is shown that a profinite CA-group has an open normal subgroup $N$ which is either virtually abelian or virtually pro-$p$ for some prime $p$. Further, a rather detailed information about the finite quotient $G/N$ is obtained.

In the present article we take a more general approach and deal with profinite CN-groups. Our results can be summarized as follows.

\begin{theorem}\label{formida} Let $G$ be a profinite CN-group, and let $F$ be the maximal normal pronilpotent subgroup of $G$. Then $F$ is open and for the finite quotient $G/F$ one of the following occurs.
\begin{enumerate}
\item $G/F$ is cyclic.
\item $G/F$ is a direct product of a cyclic group of odd order and a (generalized) quaternion group.
\item $G/F$ is a Frobenius group with cyclic kernel of odd order and cyclic complement. In this case $F$ is pro-$p$ for some prime $p$.
\item $G/F$ is isomorphic to the group $SL(2,3)$. In this case $F$ is nilpotent and the order of $F$ is divisible by at least two primes one of which is $2$.
\item $G/F$ is almost simple and $F$ is a pro-$2$ group.
\end{enumerate}
\end{theorem}

Recall that the group $SL(2,3)$ has order 24 and is isomorphic to a semidirect product of the quaternion group $Q_8$ by the cyclic group of order 3 which acts on $Q_8$ nontrivially. Recall also that a group is almost simple if it contains a non-abelian simple group and is contained within the automorphism group of that simple group.

An immediate corollary of the above theorem is that the prosoluble radical in a profinite CN-group either is the whole group or is a pro-2 group. For finite CN-groups this fact was established by Suzuki \cite{suzu61}. 

In the end of the article we give explicit examples of profinite CN-groups showing that indeed none of the alternatives mentioned in Theorem \ref{formida} can be omitted. It is worth mentioning that the example where $G/F$ is almost simple and $F$ is a pro-$2$ group is not finitely generated. We have no reason to suspect that every finitely generated infinite profinite CN-group is necessarily prosoluble. Therefore the following problem is natural.

\begin{problem}\label{111} Find an example of a finitely generated infinite profinite CN-group which is not prosoluble.
\end{problem}

We remark that all known examples of infinite profinite groups with abelian centralizers are prosoluble. The problem of finding a non-prosoluble one was raised in \cite{ppz}.

The next section contains (mostly well-known) results on automorphisms of finite and profinite groups. In Section 3 we show that profinite CN-groups are virtually pronilpotent. Section 4 contains a number of useful lemmas clarifying the structure of profinite CN-groups. Theorem \ref{formida} is established in Section 5. The notation used in this paper is standard.

\section{Automorphisms of finite and profinite groups}
Throughout this paper, every homomorphism of profinite groups is continuous, and every subgroup is closed. A cyclic group is always finite.

Many results of the theory of finite groups admit natural interpretation for profinite groups. This can be exemplified by the Sylow theorems, the Frattini argument, and so on. Throughout the article we use certain profinite versions of facts on finite groups without explaining in detail how the results on profinite groups can be deduced from the corresponding ones on finite groups. On all such occasions the deduction can be performed via the routine inverse limit argument. 

In the present auxiliary section we concentrate on automorphisms of finite and profinite groups. If $A$ is a group of automorphisms of a group $G$, the subgroup generated by elements of the form $g^{-1}g^\alpha$ with $g\in G$ and $\alpha\in A$ is denoted by $[G,A]$. It is well-known that the subgroup $[G,A]$ is an $A$-invariant normal subgroup in $G$. We also write $C_G(A)$ for the centralizer  of $A$ in $G$ and $A^\#$ for the set of nontrivial elements of $A$. Most of results given here were also discussed in \cite{ppz}.

Our first two lemmas provide a list of well-known facts on coprime  actions (see for example \cite[Ch.~5 and 6]{go}).  Here $|G|$ means the order of the profinite group $G$ (see for example \cite{rz}). A detailed proof of the profinite version of item (iii) of Lemma \ref{cc} can be found in \cite{Shumyatsky:98}. A proof of Lemma \ref{ccc} can be found in \cite{ppz}.

\begin{lemma}\label{cc}
Let  $A$ be a finite group of automorphisms of a profinite group $G$ such that $(|G|,|A|)=1$. Then
\begin{enumerate}
\item[(i)] $G=[G,A]C_{G}(A)$.
\item[(ii)] $[G,A,A]=[G,A]$. 
\item[(iii)] $C_{G/N}(A)=C_G(A)N/N$ for any $A$-invariant normal subgroup $N$ of $G$.
\item[(iv)] If $G$ is pronilpotent and $A$ is a noncyclic abelian group, then $G=\prod_{a\in A^{\#}}C_{G}(a)$.
\end{enumerate}
\end{lemma}

\begin{lemma}\label{ccc}
Let  $\alpha$ be an automorphism of a finite group $G$ such that $(|G|,|\alpha|)=1$. 
\begin{enumerate}
\item[(i)] If $G$ is cyclic of $2$-power order, then $\alpha=1$.
\item[(ii)] If $G$ is cyclic of prime-power order, then either $\alpha=1$ or $C_G(\alpha)=1$.
\item[(iii)] If $G$ is a (generalized) quaternion group, then either $\alpha=1$ or $|\alpha|=3$ and $|G|=8$.
\end{enumerate}
\end{lemma}

The next lemma is Lemma 2.3 in \cite{ppz}.
\begin{lemma}\label{rankexpo}
If $A$ is a noncyclic group of order $p^2$ acting 
on an additive abelian group $G$,
then $pG\subseteq\sum_{a\in A^\#} C_G(a)$.
\end{lemma}

Recall that a finite group $G$ is called a group of Frobenius if $G$ is a product of its normal subgroup $K$ (called kernel) and a nontrivial subgroup $H$ (called complement) such that $H\cap H^g=1$ for each $g\in G$. The reader is referred to \cite{rz} for the theory of profinite Frobenius groups. The following lemma is immediate from \cite[Lemma 2.4]{khumashu}.
\begin{lemma}\label{fro}
Suppose that a profinite group $G$ admits a finite Frobenius group of automorphisms $FH$ 
with kernel $F$ and complement $H$ such that $C_G(F)=1$ and $(|G|,|F|)=1$.
Then $G=\langle C_G(H)^f\ \vert\ f\in F\rangle$.
\end{lemma}

Recall that an automorphism $\alpha$ of order $p$ of a group $G$ is {\it splitting} if $$x\cdot x^\alpha\cdot x^{\alpha^2}\cdot\dots\cdot x^{\alpha^{p-1}}=1$$ for all $x\in G$. Obviously, the automorphism $\alpha$ is splitting if and only if $x\alpha$ has order $p$ in the semidirect product $G\langle\alpha\rangle$ for all $x\in G$. Khukhro proved that if a $d$-generator nilpotent group $G$ admits a splitting automorphism of prime order $p$, then $G$ is nilpotent of class bounded in terms of $d$ and $p$ only \cite{khu61} (see also \cite[Theorem 7.2.1]{khu2}). The following theorem is an immediate corollary of Khukhro's result. Recall that a group is said to locally have some property if every finitely generated subgroup has that property.

\begin{theorem}\label{khukhro} Let $p$ be a prime and  $G$ a pro-$p$ group admitting a splitting automorphism  of order $p$. Then $G$ is locally nilpotent.
\end{theorem}

We conclude this section by stating the  profinite version of the result which   is a combination of famous results of  Thompson \cite{Thompson:59} (saying that the group is nilpotent) and Higman  \cite{Higman:57} (bounding the nilpotency class). Its proof can be found in \cite{Shumyatsky:98} for example.

\begin{theorem}\label{higm} Let $p$ be a prime and  $G$ a pro-$p'$ group admitting a fixed-point-free automorphism  of order $p$. Then $G$ is nilpotent of class bounded by some number depending only on $p$.
\end{theorem}

\section{Profinite CN-groups are virtually pronilpotent}
In the present section we establish the part of Theorem \ref{formida} that states that profinite CN-groups are virtually pronilpotent. The following lemma is immediate from \cite[Lemma 2.8.15]{rz}.

\begin{lemma}\label{mmain} Let $G$ be a profinite group and $N$ a normal open subgroup of $G$. There exists a subgroup $H$ of $G$ such that 
$G=NH$ and $N\cap H$ is pronilpotent.
\end{lemma}

In what follows $\pi(G)$ denotes the set of prime divisors of the order of a profinite group $G$. An element $x\in G$ is a $\pi$-element (for a set of primes $\pi$) if $\pi(\langle x\rangle)\subseteq\pi$, where $\langle x\rangle$ denotes the procyclic subgroup generated by $x$. An element $x\in G$ is a $\pi'$-element if $\pi(\langle x\rangle)\cap\pi=\emptyset$. We say that two elements $x,y$ of $G$ have coprime orders if $\pi(\langle x\rangle)\cap\pi(\langle y\rangle)=\emptyset$. We write $O_p(G)$ to denote the maximal normal pro-$p$ subgroup of $G$.

The proof of the next theorem uses the fact that a profinite group $G$ is pronilpotent if and only if any two elements of $G$ of coprime orders commute. In the sequel this fact will be used throughout the article without being explicitly mentioned. 
\begin{theorem} \label{32}
A torsion-free profinite CN-group is pronilpotent. 
\end{theorem}
\begin{proof} By way of contradiction suppose that a torsion-free profinite CN-group $G$ is not pronilpotent. Choose a normal open subgroup $N$ of $G$ such that $G/N$ is not nilpotent. By Lemma \ref{mmain} there exists a subgroup $H$ of $G$ such that $G=NH$ and $N\cap H$ is pronilpotent. Evidently, $H$ is not pronilpotent. We can choose two elements $x,y\in H$ which have coprime orders and do not commute. Set $X=\langle x\rangle\cap N$ and $Y=\langle y\rangle\cap N$. Since $G$ is torsion-free, both $X$ and $Y$ are nontrivial. Because $N\cap H$ is pronilpotent, $X$ and $Y$ commute. We see that $\langle x,Y\rangle\leq C_G(X)$ and $\langle X,y\rangle\leq C_G(Y)$. Taking into account that $C_G(X)$ and $C_G(Y)$ are pronilpotent deduce that $x$ and $y$ commute, contrary to the choice of $x$ and $y$. The proof is complete.
\end{proof}

\begin{lemma}\label{eleme} Let $p$ be a prime and $B$ a noncyclic subgroup of order $p^2$ normalizing a pro-$p'$ subgroup $Q$ of a profinite CN-group $G$. Then $[Q,B]=1$.
\end{lemma}
\begin{proof} Let $b\in B^{\#}$. Observe that $\langle B,C_Q(b)\rangle\leq C_G(b)$ and therefore $C_Q(b)\leq C_Q(B)$. By {Lemma~\ref{cc}}(iv), $Q=\prod_{b\in B^{\#}}C_Q(b)$ and so $Q=C_Q(B)$, as required.
\end{proof} 
The next lemma provides an important technical tool for dealing with profinite CN-groups.

\begin{lemma}\label{tec} 
Let $\pi$ be a set of primes and $G$ a profinite CN-group in which an infinite abelian pro-$\pi'$ subgroup $A$ normalizes a pro-$\pi$ subgroup $P$. Then $[P,A]=1$.
\end{lemma}
\begin{proof} 
Without loss of generality we assume that $G=PA$. Note that since $A$ is abelian, $C_P(a)=C_P(A)$ for any nontrivial $a\in A$. Assume that $[P,A]\neq1$.

Let us show that there exists a bound on the order of torsion elements in $A$. Suppose that this is false. Let $N$ be a normal subgroup of $G$, which is contained in $P$ as an open subgroup. Suppose $N$ has index $n$ in $P$. 
Pick a torsion element $a\in A$ whose order is at least $n!$. It is clear that a nontrivial power of $a$, say $a^i$, acts trivially on $P/N$. {Lemma~\ref{cc}}(iii) shows that $P=NC_P(a^i)$. Hence, $P=NC_P(A)$. This happens for every normal subgroup $N$ of $G$ which is contained in $P$ as an open subgroup. Therefore  $P=C_P(A)$ and so $[P,A]=1$. Thus, indeed there exists a bound on the order of torsion elements in $A$.

Suppose now that $A$ contains a noncyclic subgroup $B$ of order $q^2$ for some prime~$q$. By Lemma \ref{eleme} $[P,B]=1$ and so again $P=C_P(A)$, i.e. $[P,A]=1$. Therefore without loss of generality we may assume that $A$ has no noncyclic finite subgroups. Hence, torsion subgroups in $A$ are cyclic of bounded order. It follows that $A$ has an open torsion-free subgroup, say $A'$. Obviously, it is sufficient to show that $[P,A']=1$ and, replacing  $A$ by $A'$ if necessary,  we assume without loss of generality that $A$ is torsion-free.

Since $[P,A]\neq1$, the group $G$ has a finite quotient $G/N$ in which the images of $P$ and $A$ do not commute. {Lemma~\ref{mmain}} shows that $G$ has a subgroup $H$ such that $G=NH$ and $N\cap H$ is pronilpotent. Set $P_1=P\cap H$ and let $A_1$ be a $\pi'$-Hall subgroup of $H$. Of course, $A_1$ is conjugate to a subgroup of $A$. Since the images of $P$ and $A$ in $G/N$ do not commute, we observe that both subgroups $A_1$ and $P_1$ are nontrivial. Further, since $A_1$ is torsion-free, it follows that $N\cap A_1\neq1$. Set $A_0=A_1\cap N$. Since $A_0$ is characteristic subgroup of the pronilpotent group $N\cap H$, it is normal in $H$  and moreover $A_0$ normalizes $P_1$. We conclude that  $A_0$ centralizes $P_1$. Hence, $\langle A_1,P_1\rangle\leq C_G(A_0)$ and therefore $P_1\leq C_G(A_1)$. This implies that the images of $P$ and $A$ in $G/N$ commute, a contradiction.
\end{proof}

\begin{lemma}\label{golova} Let $\pi$ be a set of primes and $G$ a profinite CN-group. Suppose that a pro-$\pi'$ subgroup $A$ normalizes a pro-$\pi$ subgroup $P$ of $G$. Then $A/C_A(P)$ is finite with Sylow subgroups cyclic or (generalized) quaternion.
\end{lemma}

\begin{proof} Set $C=C_A(P)$. Observe that in view of Lemma \ref{tec} each infinite procyclic subgroup of $A$ belongs to $C$. Therefore $A/C$ is torsion and, by Zelmanov's theorem \cite{zelm}, locally finite. Suppose that $A/C$ contains a noncyclic subgroup $B/C$ of order $q^2$ for some prime $q$. The group $B/C$ acts faithfully on $P$. By {Lemma~\ref{cc}}(iv), $P=\prod_{bC\in (B/C)^{\#}}C_{P}(bC)$. Lemma \ref{eleme} shows that $C\neq1$.

Suppose that for some $b\in B\setminus C$ the centralizer $C_C(b)$ is nontrivial. Choose a nontrivial element $x\in C_C(b)$ and observe that $P\langle b\rangle\leq C_G(x)$. It follows that $P\langle b\rangle$ is pronilpotent whence $[P,b]=1$ and so $b\in C$. This is a contradiction and therefore $C_C(b)=1$ for each $b\in B\setminus C$. It follows that the order of each $b\in B\setminus C$ is precisely $q$.

Now Khukhro's Theorem \ref{khukhro} states that $C$ is locally nilpotent. Choose $b_1,b_2\in B\setminus C$ such that $B=\langle b_1,b_2,C\rangle$. For some nontrivial $x\in C$ set $X=\langle x,b_1,b_2\rangle\cap C$. We observe that $X$ has finite index in $\langle x,b_1,b_2\rangle$ and so $X$ is finitely generated. Hence $X$ is nilpotent. Let $Z=Z(X)$. Observe that $\langle b_1,b_2\rangle$ normalizes $Z$ and the induced group (denote it by $\bar{B}$) of automorphisms of $Z$ is a finite noncyclic group of order $q^2$. Since $C_C(b)=1$ for each $b\in B\setminus C$, it follows that $Z$ has exponent $q$ (Lemma \ref{rankexpo}). Choose a nontrivial element $z\in Z$ and consider the action of $\bar{B}$ on $Y=\langle z^{\bar{B}}\rangle$. Since both groups $\bar{B}$ and $Y$ have finite $q$-power order, the centralizer $C_Y(\bar{B})$ is nontrivial. It follows that also the centralizers $C_C(b_1)$ and $C_C(b_2)$ are nontrivial. This contradicts the assumption that $C_C(b)=1$ for each $b\in B\setminus C$.

Hence, for any prime $q$ subgroups of order $q^2$ of $A/C$ are cyclic. Thus, Sylow subgroups of $A/C$ are cyclic or generalized quaternion (see \cite[Theorem 5.4.10(ii)]{go}). It remains to show that $A/C$ is finite. The finiteness of $A/C$ is immediate from Herfort's theorem that the set $\pi(H)$ is finite for any profinite torsion group $H$ (see \cite{herfort}). The lemma is established.
\end{proof}

Recall that the Fitting height of a finite soluble group $G$ is the length $h(G)$ of a shortest series of normal subgroups all of whose quotients are nilpotent. By the Fitting height of a prosoluble group $G$ we mean the length $h(G)$ of a shortest series of normal subgroups all of whose quotients are pronilpotent. Note that in general a prosoluble group does not necessarily have such a series. The parameter $h(G)$ is finite if, and only if, $G$ is an inverse limit of finite soluble groups of bounded Fitting height.

\begin{lemma}\label{44} 
Let $p$ be a prime and $G$ a finite soluble group in which for every prime $q\neq p$ the Sylow $q$-subgroups are cyclic or generalized quaternion. The Fitting height of $G$ is at most $4$.
\end{lemma}
\begin{proof} Suppose first that $O_p(G)=1$. Let $M=F(G)$ and $M'$ be the commutator subgroup of $M$. Set $\bar{G}=G/M'$ and $\bar{M}=M/M'$. By \cite[Theorem 6.1.6]{go}, $\bar{M}=F(\bar{G})$ and $C_{\bar{G}}(\bar{M})=\bar{M}$. It follows that the quotient $G/M$ embeds in the group of automorphisms of $\bar{M}$. Note that $\bar{M}$ is either cyclic or direct product of a cyclic group of odd order and a noncyclic group of order 4. Recall that the group of automorphisms of a cyclic group is abelian while that of the noncyclic group of order 4 is isomorphic to the non-abelian group of order 6. We conclude that $G/M$ is metabelian. Thus, if $O_p(G)=1$, then $h(G)\leq3$.

Now drop the assumption that $O_p(G)=1$. The above paragraph shows that $h(G/O_p(G))\leq3$. Thus, $h(G)\leq4$.
\end{proof}

\begin{lemma}\label{55} Let $p$ be a prime and $G$ a prosoluble CN-group. Let $P$ be a pro-$p$ subgroup of $G$ and $T=N_G(P)$. The Fitting height of $T$ is at most $5$.
\end{lemma}
\begin{proof} Since the centralizers in $G$ are pronilpotent, it follows that $C_T(P)$ is contained in the Fitting subgroup $F(T)$. Lemma \ref{golova} states that $A/C_A(P)$ is finite with Sylow subgroups cyclic or generalized quaternion for each pro-$p'$ subgroup $A\leq T$. Thus, every finite continuous image of $T/F(T)$ satisfies the hypothesis of Lemma \ref{44}. Hence $T/F(T)$ is an inverse limit of finite soluble groups of Fitting height at most 4. It follows that $h(T/F(T))\leq4$ and so $h(T)\leq5$.
\end{proof}

Recall that a Sylow basis in a group $G$ is a family of pairwise permutable Sylow $p_i$-subgroups $P_i$ of $G$, exactly one for each prime. The \emph{basis normalizer} of such Sylow basis in $G$ is $\bigcap_i N_G(P_i)$. This subgroup is also known under the name of system normalizer. If $G$ is a profinite group and $T$ is a basis normalizer in $G$, then $T$ is pronilpotent and $G=\gamma_\infty(G)T$, where $\gamma_\infty(G)$ denotes the intersection of the terms of the lower central series of $G$. Furthermore, every prosoluble group $G$ possesses a Sylow basis and any two basis normalizers in $G$ are conjugate (see \cite[Prop.~2.3.9]{rz} and \cite[9.2]{rob}). 

\begin{lemma}\label{fh5} Let $G$ be a prosoluble CN-group. The Fitting height of $G$ is at most 5.
\end{lemma}
\begin{proof} Set $K_1=\gamma_\infty(G)$ and $K_{i+1}=\gamma_\infty(K_i)$ for $i=1,2,\dots$. Assume that the lemma is false and $K_5\neq1$.
Let $\{P_1,P_2,\dots\}$ be a Sylow basis of $G$, and let $T_1$ be the basis normalizer corresponding to $\{P_1,P_2,\dots\}$. We have $G=K_1T_1$. 

For each $i=1,2,\dots$ consider the Sylow basis $\{P_1\cap K_i,P_2\cap K_i, \dots\}$ in $K_i$ and let $T_i\leq K_i$ be the corresponding basis normalizer. The specific choice of the Sylow bases $\{P_1\cap K_i,P_2\cap K_i,\dots\}$ guarantees that $T_j$ normalizes $T_k$ whenever $j\leq k$. Note that we have the equalities $$G=K_1T_1=K_2T_2T_1=\dots=K_iT_iT_{i-1}\cdots T_2T_1=\dots.$$ In particular, $G=K_6T$ , where $T=T_6T_5T_4T_3T_2T_1$. Since $K_6\neq K_5$, it follows that $G/K_6$ has Fitting height 6 and therefore $h(T)=6$. This contradicts Lemma \ref{55} since $T$ normalizes a pro-$p$ subgroup of $K_6$.
\end{proof}

Denote by $\mathfrak C$ the class of all finite groups whose soluble subgroups are of Fitting height at most 5. Obviously, $\mathfrak C$ is closed under taking subgroups. It was shown in \cite{ppz} that the class $\mathfrak C$ is also closed under taking quotients. Therefore we have the following lemma (cf. Lemma 3.8 in \cite{ppz}).

\begin{lemma} Any profinite CN-group is pro-$\mathfrak C$.
\end{lemma}

Recall that the nonsoluble length $\lambda(G)$ of a finite group $G$ is defined as the minimum number of nonsoluble factors in a normal series each of whose factors is either soluble or a nonempty direct product of non-abelian simple groups. It was shown in \cite[Cor.~1.2]{KhSh:15} that the nonsoluble length of a finite group $G$ does not exceed the maximum Fitting height of soluble subgroups of $G$. It follows that for any group $G$ in $\mathfrak C$ we have $\lambda(G)\leq5$. We conclude that each group $G$ in $\mathfrak C$ has a characteristic series of length at most 35 each of whose factors is either nilpotent or a direct product of non-abelian simple groups. More precisely, each group $G$ in $\mathfrak C$ has a characteristic series
\[1=G_0\leq G_1\leq\dots\leq G_{35}=G\]
such that the factors $G_6/G_5$, $G_{12}/G_{11}$, $G_{18}/G_{17}$, $G_{24}/G_{23}$, $G_{30}/G_{29}$ are direct product of non-abelian simple groups while the other factors are nilpotent. Important results of Wilson \cite[Lemma~2 and Lemma~3]{Wilson:83} now guarantee that any pro-$\mathfrak C$ group has a characteristic series of length at most 35 each of whose factors is either pronilpotent or a Cartesian product of non-abelian simple groups. 
Thus, we have proved the following result.

\begin{lemma}\label{series} Any profinite CN-group has a characteristic series of length at most 35 each of whose factors is either pronilpotent 
or a Cartesian product of non-abelian finite simple groups.
\end{lemma}

We remark that our final results show that the characteristic series in the above lemma can be chosen of length at most 3 with the first term being an open pronilpotent subgroup. 

\begin{lemma}\label{metanilp} Let $G$ be a prosoluble CN-group. Then $G$ is virtually pronilpotent.
\end{lemma}
\begin{proof} By Lemma \ref{fh5} $h(G)\leq5$. Let us use induction on $h(G)$. Set $R=\gamma_\infty(G)$. By induction, $R$ is virtually pronilpotent. Therefore $R$ contains an open characteristic pronilpotent subgroup $N$. Let $H={\{x\in G\ \vert\ [R,x]\leq N\}}$. Since $R/N$ is finite, it is clear that $H$ is an open subgroup in $G$. Note that the image of $R\cap H$ in $H/N$ is central. Therefore $H/N$ is pronilpotent. It is sufficient to show that $H$ is virtually pronilpotent. 

Since $H$ is CN, we have $C_H(P)\leq F(H)$ for any Sylow subgroup $P$ of $N$. Lemma \ref{golova} shows that the Hall $p'$-subgroup of $H/F(H)$ is finite for each $p\in\pi(N)$. If $N$ is pro-$p$, set $P=O_p(H)$ and observe that $H/P$ is pro-$p'$ and so $H/P$ is finite. If $N$ is not pro-$p$, choose two primes $p,q\in\pi(N)$ and let $P,Q$ be $p$-Sylow and $q$-Sylow subgroups of $N$. Since both $p'$-Hall and $q'$-Hall subgroups of $H/F(H)$ are finite, the quotient $H/F(H)$ is finite, too.
\end{proof} 

\begin{theorem}\label{main} A profinite CN-group is virtually pronilpotent.
\end{theorem}
\begin{proof} Let $G$ be a profinite CN-group. By {Lemma~\ref{series}}, $G$ has a normal series of length at most 35
\[1= G_0\leq\dots\leq G_l=G\] 
such that each factor $G_{i+1}/G_i$ is either pronilpotent 
or a Cartesian product of non-abelian finite simple groups.
Let $l$ be the minimum of lengths of such series. 
If $l=1$, then $G$ is either pronilpotent or finite.
We therefore assume that $l\geq2$ and use induction on $l$.
Set $R=G_{l-1}$ and observe that by induction $R$ is virtually pronilpotent. Of course, we can assume that $R$ is infinite and therefore $F(R)\neq1$.

Let $N$ be the maximal open normal pronilpotent subgroup in $R$. As in the proof of Lemma \ref{metanilp}, let $H={\{x\in G\ \vert\ [R,x]\leq N\}}$.
Since $R/N$ is finite, it is clear that $H$ is an open subgroup in $G$.
Note that the image of $R\cap H$ in $H/N$ is central. It is sufficient to show that $H$ is virtually pronilpotent. Thus, without loss of generality we can assume that $R/N$ is central in $G/N$.

Suppose first that $G/R$ is pronilpotent. Since $R/N$ is central in $G/N$, it follows that $G/N$ is pronilpotent and the theorem is immediate from {Lemma~\ref{metanilp}}.

Therefore it remains to deal with the case where $G/R$ is a Cartesian product of finite non-abelian simple groups. In that case $G$ is a product of its normal subgroups $S_i$, where $i\in I$, such that $S_i/R$ is a simple direct factor of $G/R$. Assume that the set of indices $I$ is infinite. Let $P$ be a Sylow $p$-subgroup of the pronilpotent subgroup $N$. Since $G$ is CN, we have $C_G(P)\leq F(G)$. For each $i\in I$ choose a $p'$-element $a_i\in S_i \setminus R$. Note that the image of the subgroup $A=\langle a_i\ \vert\ i\in I\rangle$ in $G/N$ is infinite and nilpotent of class at most two. In particular, $A$ is prosoluble. Let $B$ be a Hall $p'$-subgroup of $A$. Lemma \ref{golova} shows that the image of $B$ in $G/F(G)$ is finite. Then it follows that also the image of $A$ in $G/F(G)$ is finite. This yields a contradiction since the image of $A$ in $G/N$ is infinite while $N$ has finite index in $F(G)$.
\end{proof}
\section{Useful lemmas}

Let $G$ be a profinite CN-group. We will write $F$ for $F(G)$ and $\pi$ for $\pi(F)$. Of course, whenever $p\in\pi$ the Sylow $p$-subgroup of $F$ is precisely $O_p(G)$. Theorem \ref{main} shows that $G/F$ is finite. It is straightforward from Lemma \ref{golova} that if $F$ is not pro-$p$, then Sylow subgroups of $G/F$ are either cyclic or generalized quaternion. Moreover if $F$ is pro-$p$ for some prime $p$, then for any prime $q\neq p$ Sylow $q$-subgroups of $G/F$ are either cyclic or generalized quaternion. 

\begin{lemma}\label{ooo} Assume that $\pi$ contains at least two primes.  Let $p\in\pi$ and $P=O_p(G)$. Suppose that $a\in G$ a $p$-element such that $a\not\in P$. Then $C_P(a)=1$. In particular, $\langle a\rangle\cap F=1$.
\end{lemma} 
\begin{proof} Suppose that $P_0=C_P(a)\neq1$. Observe that since $\pi$ contains at least two primes, there is a nontrivial normal pronilpotent pro-$p'$ subgroup $Q$ in $G$. The centralizer $C_G(P_0)$ contains both $a$ and $Q$. Taking into account that $C_G(P_0)$ is pronilpotent we deduce that $a$ centralizes $Q$. Since $Q$ is normal and $C_G(Q)$ is pronilpotent, it follows that $a\in F$ and more precisely $a\in P$. This is a contradiction that shows that $C_P(a)=1$. In particular, we conclude that $\langle a\rangle\cap P=1$.
\end{proof}

\begin{lemma}\label{hah1} Suppose that $\pi$ contains at least two primes. Let $p$ be a prime contained in $\pi\cap\pi(G/F)$. Then $O_p(G)$ is a torsion-free locally nilpotent group.  
\end{lemma}
\begin{proof} Set $P=O_p(G)$ and choose a $p$-element $a\in G$ such that $a\not\in P$. In view of Lemma \ref{ooo} for any $x\in P$ the order of $ax$ is finite and equals that of $a$. Let $b$ be an element of order $p$ in $\langle a\rangle$. We see that $bx$ has order $p$ for each $x\in P$. Khukhro's Theorem \ref{khukhro} states that $P$ is locally nilpotent. It remains to show that $P$ is torsion-free. 

By contradiction, assume that $P$ contains a nontrivial element $y$ of finite order. Since $P$ is locally nilpotent, it follows that the subgroup $\langle a,y\rangle$ is finite and nilpotent. This leads to a contradiction because $C_P(a)=1$. The proof is complete. 
\end{proof}

\begin{lemma}\label{hah2} Let $P$ be a locally nilpotent pro-$p$ subgroup of $G$ and $a$ a $p'$-element such that $C_P(a)\neq1$. Then $[P,a]=1$.
\end{lemma}
\begin{proof} Let $1\neq x\in C_P(a)$ and choose an arbitrary element $y\in P$. Since $\langle x,y\rangle$ is nilpotent, there is a nontrivial element $z\in Z(\langle x,y\rangle)$. We see that $a,z\in C_G(x)$. Since $C_G(x)$ is pronilpotent, $a$ and $z$ commute. Thus $a,y\in C_G(z)$. Since $C_G(z)$ is pronilpotent, $a$ and $y$ commute. So $a$ centralizes an arbitrary element $y$ of $P$. The result follows.
\end{proof}  

\begin{lemma}\label{hap3} Suppose that $\pi$ contains at least two different primes. Then nontrivial subgroups of $G/F(G)$ are not Frobenius.
\end{lemma}
\begin{proof} Suppose that the lemma is false and choose a counter-example with  $|G/F(G)|$ as small as possible. Then $G/F(G)$ is a Frobenius group. Since the Sylow subgroups of $G/F(G)$ are cyclic, or generalized quaternion, the Frobenius kernel and complement of $G/F(G)$ are of prime order, say of order $r$ and $s$, respectively. Further, by Lemma \ref{ooo} elements of prime-power order in $G\setminus F$ have finite order. Therefore $G$ has a finite Frobenius subgroup $KH$ with kernel $K$ of prime order $r$ and complement $H$ of prime order $s$ such that $G=FKH$ and $F\cap KH=1$. Let $p,q\in\pi$, and let $P$ and $Q$ be $p$-Sylow and $q$-Sylow subgroups of $F$.

Suppose first that $C_P(K)\neq1$. Then $QK$ is pronilpotent because $Q,K\leq C_P(K)$. If $r\neq q$, then $K\leq C_G(Q)\leq F$ which is a contradiction. 
Therefore $r=q$ and, by Lemma \ref{hah1}, $Q$ is locally nilpotent. Lemma \ref{hah2} shows that $C_Q(H)=1$ and therefore, because of Lemma \ref{fro}, $C_Q(K)\neq1$. Observe that $PK$ centralizes $C_Q(K)$ and so $[P,K]=1$, a contradiction.

Thus, $C_P(K)=1$. By a symmetric argument $C_Q(K)=1$. It follows from Lemma \ref{fro} that $C_P(H)\neq1$ and $C_Q(H)\neq1$. We deduce that both $PH$ and $QH$ are pronilpotent. Recall that $H$ has prime order. Therefore $H$ centralizes at least one of the subgroups $P$ and $Q$, whence $H\leq F$. This is a contradiction.
\end{proof}

Let $p$ be a prime. A normal subgroup $N$ of a finite group $K$ is a normal $p$-complement if $N=O_{p'}(K)$ and $K/N$ is a $p$-group. The well-known theorem of Frobenius states that $K$ possesses a normal $p$-complement if and only if $N_K(H)/C_K(H)$ is a $p$-group for every nontrivial $p$-subgroup $H$ of $K$ (see \cite[7.4.5]{go}).

The next lemma provides a sufficient condition under which $G/F$ is isomorphic to $SL(2,3)$. Recall that the group $SL(2,3)$ has order 24 and is isomorphic to a semidirect product of the quaternion group $Q_8$ by the cyclic group of order 3 which acts on $Q_8$ nontrivially.

\begin{lemma}\label{hah4} Suppose that $\pi$ contains at least two primes and $\pi\cap\pi(G/F)$ is non-empty. Assume that $G/F$ is not nilpotent. Then  $G/F\cong SL(2,3)$.
\end{lemma}
\begin{proof}  Set $\bar{G}=G/F$ and write $\bar{X}$ for the image of a subgroup $X$ in $\bar{G}$. We know that  $\bar{G}$ is finite with Sylow subgroups cyclic or quaternion. Being non-nilpotent, $G/F$ does not possess a normal $p$-complement for some prime $p$. 

Suppose that $p$ is odd and so the Sylow $p$-subgroup of $\bar{G}$ is cyclic. Let $\bar{x}$ be an element of order $p$ in the Sylow $p$-subgroup. By the normal $p$-complement theorem of Frobenius there exists a $p'$-element $\bar{a}$ in $\bar{G}$ which normalizes $\langle\bar{x}\rangle$ without centralizing it. We can choose $\bar{a}$ of $q$-power order for some prime $q\neq p$. Let $\bar{b}$ be an element of order $q$ in $\langle\bar{a}\rangle$. If $\bar{b}$ does not centralize $\bar{x}$, then the subgroup $\langle\bar{b},\bar{x}\rangle$ is Frobenius which leads to a contradiction with Lemma \ref{hap3}. Hence, $\bar{b}$ centralizes $\bar{x}$. By Lemma \ref{ooo} we can choose an element $a\in G$ which maps on $\bar{a}$ and has the same order as $\bar{a}$. Of course, there is $b\in\langle a\rangle$ which maps on $\bar{b}$. We can also choose an element $x\in G$ of order $p$ which maps on $\bar{x}$. Let $P$ be a Sylow $p$-subgroup of $K=F\langle a,x\rangle$ containing $x$. Note that $FP$ is normal in $K$ and so by the Frattini argument $K=FN_K(P)$. Replacing $a$ by a conjugate we can assume that $a$ normalizes $P$. Since $\bar{b}$ centralizes $\bar{x}$, it follows that $[P,b]\leq F$ and so $P=P_1C_P(b)$, where $P_1=P\cap F$. Both $C_P(b)$ and $a$ are contained in $C_G(b)$ and so $C_P(b)$ and $a$ commute because $C_G(b)$ is pronilpotent. It follows that  $P=P_1C_P(a)$, which contradicts the fact that $\bar{a}$ normalizes $\langle\bar{x}\rangle$ without centralizing it. In particular, we have proved that $\bar{G}$ has a normal $r$-complement for each odd prime $r\in\pi(\bar{G})$.

Thus $p=2$ and there exists a $2'$-element $\bar{a}$ in $\bar{G}$ which normalizes a 2-subgroup without centralizing it. Recall that the Sylow 2-subgroup of $\bar{G}$ is either cyclic or (generalized) quaternion. In view of Lemma \ref{ccc} we deduce that $\bar{a}$ normalizes a subgroup $\bar{Q}\cong Q_8$ and $\bar{a}^3$ centralizes $\bar{Q}$. Arguing as above we can assume $a$ normalizes a pro-2 subgroup $Q$ of $G$ which maps on $\bar{Q}$.  Again, $Q=Q_1C_Q(a^3)$ where $Q_1$ is the Sylow 2-subgroup of $F$. From this (and Lemma \ref{ooo}) we deduce that $a^3=1$ and $a$ has order three. 

Let us show that $\bar{G}$ has no normal Sylow $r$-subgroups for $r\neq2$. Suppose on the contrary that for an odd prime $r$ the Sylow $r$-subgroup $\bar{R}$ is a normal in $\bar{G}$. We already know that $r\neq3$ since $\bar{a}$ does not centralize $\bar{Q}$. Let $\bar{d}$ be an element of order $r$ in $\bar{R}$. Since $\bar{R}$ is normal and $\bar{G}$ has a normal $r$-complement, it follows that $\bar{d}$ is central in $\bar{G}$. Choose $d\in G$ such that $d$ has order $r$ and maps on $\bar{d}$. Set $T=\langle FQ,a,d\rangle$. Note that $FQ\langle a\rangle$ is normal in $T$. Let $U$ be a Hall $\{2,3\}$-subgroup of $T$ containing both $Q$ and $a$. By the Frattini argument $N_T(U)$ contains a conjugate of $d$. Therefore we can assume that $d\in N_T(U)$. On the one hand, $C_U(d)$ is pronilpotent. On the other hand $C_U(d)F/F$ is isomorphic to $SL(2,3)$. This is a contradiction because $SL(2,3)$ is not nilpotent.
Thus, indeed $\bar{G}$ has no normal Sylow $r$-subgroups for $r\neq2$. Since $G$ has a normal 3-complement, it follows that the Sylow 2-subgroup of $\bar{G}$ is isomorphic to $Q_8$ and so $\bar{G}\cong SL(2,3)$, as required.
\end{proof}  
In the next lemma $F_2(G)$ stands for the second Fitting subgroup of $G$. Therefore we have $F_2(G)/F=F(G/F)$. 
\begin{lemma}\label{op1} Suppose that $\pi(F_2(G))\not\subset\pi$. Let $H$ be a $\pi'$-Hall subgroup of $F_2(G)$ and $K=N_G(H)$. Then $K/F(K)$ is cyclic. Moreover, if $K$ is not pronilpotent then both $K$ and its image in $G/F$ are Frobenius groups.
\end{lemma}
\begin{proof} By the Frattini argument, $G=FK$. Note that $H$ is a finite nilpotent group whose Sylow subgroups are either cyclic or generalized quaternion. Therefore any subgroup of $H$ contains a normal subgroup. Let $L=F(K)$. If $y\in K$ and $C_H(y)\neq1$, it follows that $y$ centralizes a nontrivial normal subgroup of $H$. In that case, since $G$ is CN, $y\in L$. Let $H_1$ be a characteristic subgroup of prime order in $H$. Since $K/C_K(H_1)$ embeds in the group of automorphisms of $H_1$ and since $C_K(H_1)\leq L$, we conclude that $K/L$ is cyclic.

Assume that $K$ is not pronilpotent. Choose $a\in K$ such that $K=L\langle a \rangle$. Let $b\in\langle a\rangle$. Suppose that $C_{L}(b)\neq1$ and, for a prime $p$, choose a $p$-element $x\in C_{L}(b)$. We already know that $C_H(b)=1$ and so $p\not\in\pi(H)$. Because $L$ is pronilpotent, it follows that $H\leq C_G(x)$. Hence $\langle b,H\rangle\leq C_G(x)$ and so $\langle b,H\rangle$ is nilpotent. This is a contradiction as $C_H(b)=1$. Therefore $C_{L}(b)=1$. Note that, since nontrivial powers of $a$ act on the finite subgroup $H$ without fixed points, $a$ has finite order. Hence, $K$ is a Frobenius group with kernel $L$ and finite cyclic complement $\langle a\rangle$.

It remains to prove that the image of $K$ in $G/F$ is a Frobenius group. It is sufficient to show that if $1\neq b\in\langle a\rangle$ and $y\in L$ such that $[y,b]\in F$, then $y\in F$. Since $a$ has finite order, without loss of generality it can be assumed that $b$ has prime order $q$ and $y$ is a $p$-element for some prime $p$. Assume that $y\not\in F$. Then $\langle y\rangle\cap F=1$. Indeed, suppose that $y^i\in F$. If $F$ is pro-$p$, then $y\in F$, a contradiction. Therefore $F$ is not pro-$p$ and $\langle y\rangle\cap F=1$ by Lemma \ref{ooo}. So without loss of generality we can assume that $y$ has prime order $p$. 

Suppose first that $p=q$. By Lemma \ref{golova} the Sylow $p$-subgroup of $G/F$ is either cyclic or generalized quaternion. Since $b$ and $y$ commute modulo $F$ and both have order $p$, we conclude that $b\in F\langle y\rangle\leq L$. This is a contradiction and therefore $p\neq q$.

Observe that $y$ normalizes $F\langle b\rangle$. Let $Q$ be the Sylow $q$-subgroup of $F\langle b\rangle$ containing $b$. In view of the Frattini argument there exists $f\in F$ such that $fy$ normalizes $Q$. The Sylow $p$-subgroup of $\langle fy\rangle$ is contained in $L$ and not in $F$. Without loss of generality we can assume that $y$ normalizes $Q$. Since $[Q,y]\leq F$, we deduce that $Q=Q_1C_Q(y)$, where $Q_1=Q\cap F$. If $Q_1=1$ then, contrary to our assumptions,  $b$ centralizes $y$. Thus, $Q_1\neq1$ and so $q\in\pi$. Suppose that also $p\in\pi$. Then $y$ centralizes $H$ and so $\langle C_Q(y),H\rangle\leq C_G(y)$, whence $[C_Q(y),H]=1$. Since $b\in Q_1C_Q(y)$, we obtain a contradiction. Therefore $p\not\in\pi$ and so $y\in H$. We note that the subgroup $H\langle b\rangle$ is a finite Frobenius group  having trivial intersection with $F$. It follows that the image of $H\langle b\rangle$ in $G/F$ is isomorphic to $H\langle b\rangle$ and therefore $[y,b]\not\in F$. This is the final contradiction. 
\end{proof}

\begin{lemma}\label{hap44} Suppose that each of the sets $\pi$ and $\pi(G/F)$ contains at least two different primes. Then $F(G)$ is nilpotent.
\end{lemma}
\begin{proof} Let $P$ be a Sylow $p$-subgroup of $F$ for some prime $p$ and suppose that $P$ is not nilpotent. It follows from Theorem \ref{higm} that if $a$ is a $p'$-element of prime order in $G\setminus F$, then $C_P(a)\neq1$. Therefore, in view of Lemma \ref{hah2}, we deduce that  $P$ is not locally nilpotent. Let $Q$ be a Sylow $q$-subgroup of $F$ for some prime $q\neq p$. Both $a$ and $Q$ are contained in the centralizer of $C_P(a)$. Since $a$ does not centralize $Q$, we conclude that $a$ is of order $q$. Taking into account that $\pi(G/F)$ contains at least two different primes, we can choose a $q'$-element $b$ of prime order. Note that, by virtue of Lemma \ref{hah1}, $b$ cannot be of order $p$ because $P$ is not locally nilpotent. Repeating the above argument with $a$ replaced by $b$, we deduce that $b$ is of order $q$, contrary to the choice of $b$. Thus, we have shown that each Sylow subgroup of $F$ is nilpotent. 

Now, in the case where $\pi$ is finite the result is immediate. So assume that $\pi$ is infinite. As above, we assume that $a$ is an element of prime order $q$ in $G\setminus F$. Let $S$ be the $q'$-Hall subgroup of $F$. If $C_S(a)=1$, then, by Theorem \ref{higm}, $S$ is nilpotent and the lemma follows. Hence, we assume that $C_S(a)\neq1$. Choose a prime $r\in\pi(C_S(a))$ and a nontrivial $r$-element $x$ in $C_S(a)$. Let $S_1$ be the Hall $r'$-subgroup in $S$. Note $S_1$ is normal in $G$ and is nontrivial since $\pi$ is infinite. Further, we observe that both $a$ and $S_1$ are contained in the pronilpotent subgroup $C_G(x)$. So we deduce that $a$ centralizes $S_1$ and therefore $a\in F$. This contradiction shows that $C_S(a)=1$. The proof is complete.
\end{proof}
\section{The structure of profinite CN-groups}

We are now ready to finalize the proof of Theorem \ref{formida}. This section is divided in three parts. In the first one we handle prosoluble CN-groups (see Theorem \ref{mainso}). Subsection 5.2 deals with the non-prosoluble case of Theorem \ref{formida} (see Theorem \ref{222}). In the final subsection we give examples of profinite CN-groups showing that indeed no alternative mentioned in Theorem \ref{formida} can be omitted. It is easy to see that the combination of Theorems \ref{main}, \ref{mainso}, and \ref{222} is precisely Theorem \ref{formida}.
\subsection{On prosoluble CN-groups}

\begin{theorem}\label{mainso} Let $G$ be a prosoluble CN-group, and let $F$ be the maximal normal pronilpotent subgroup of $G$. Then one of the following holds.
\begin{enumerate}
\item $G/F$ is cyclic.
\item $G/F$ is a direct product of a cyclic group of odd order and a generalized quaternion group.
\item $G/F$ is Frobenius with cyclic kernel of odd order and cyclic complement. In this case $F$ is pro-$p$ for some prime $p$.
\item $G/F(G)\cong SL(2,3)$. In this case $F(G)$ is nilpotent and $\pi(F(G))$ has at least two primes one of which is $2$. 
\end{enumerate}
\end{theorem}
\begin{proof} We know that $G/F$ is finite with Sylow subgroups either cyclic or generalized quaternion. Thus, if $G/F$ is nilpotent, then it is either cyclic or direct product of a cyclic group of odd order and a generalized quaternion group. Suppose that $G/F$ is not nilpotent.

Consider first the case where $\pi(F)$ has at least two primes and $\pi(F)\cap\pi(G/F)$ is non-empty. Then $G/F\cong SL(2,3)$ by Lemma \ref{hah4}. Moreover, by Lemma \ref{hap44}, $F$ is nilpotent. Suppose that $2\not\in\pi(F)$. Let $S$ be a Sylow 2-subgroup of $G$. Observe that $S$ is quaternion and let $a$ be the unique involution in $S$. By the Frattini argument $G=FK$, where $K=N_G(S)$. Note that $K$ is not pronilpotent since the image of $K$ in $G/F$ is isomorphic to $SL(2,3)$. On the other hand, $K\leq C_G(a)$ and so $K$ must be pronilpotent. This is a contradiction. Hence, in the case where $\pi(F)$ has at least two primes and $\pi(F)\cap\pi(G/F)$ is non-empty we have $G/F\cong SL(2,3)$ and $F$ is nilpotent with $2\in\pi(F)$.

Now suppose that $\pi(F)$ has at least two primes and $\pi(F)\cap\pi(G/F)=\emptyset$. By Lemma \ref{op1} $G/F$ is a Frobenius group. However this contradicts Lemma \ref{hap3}. Thus, the case where $\pi(F)$ has at least two primes and $\pi(F)\cap\pi(G/F)=\emptyset$ does not occur.

It remains to handle the case where $F$ is pro-$p$ for some prime $p$. Certainly, $F_2(G)$ is not pro-$p$ and let $H$ be a Hall pro-$p'$ subgroup of $F_2(G)$. By the Frattini argument $G=FN_G(H)$. Since $G/F$ is not nilpotent, it follows that $N_G(H)$ is not pronilpotent. By Lemma \ref{op1}, $G/F$ is a Frobenius group. The proof is now complete. 
\end{proof}

\subsection{On non-prosoluble CN-groups}

\begin{lemma}\label{frobe} Every finite non-soluble group has a dihedral subgroup which is a Frobenius group.
\end{lemma}
\begin{proof} Let $G$ be a finite non-soluble group. There is an involution $a\in G$ such that $a\not\in F(G)$. By the Baer-Suzuki theorem \cite[Theorem 3.8.2]{go} the commutator $[x,a]$ has odd order for some $x\in G$. The subgroup $\langle a,[x,a]\rangle$ is a dihedral subgroup which is a Frobenius group.
\end{proof}

Suzuki proved in \cite{suzu61} that a finite CN-group having a normal subgroup of odd order is soluble. Our next theorem provides a profinite analog of Suzuki's result.
\begin{theorem}\label{odd} A profinite CN-group containing a nontrivial normal pro-$2'$ subgroup is prosoluble.
\end{theorem}

\begin{proof} Suppose that the theorem is false and let $G$ be a counter-example with $|G/F(G)|$ as small as possible (recall that $G/F(G)$ is finite by Theorem \ref{main}). Let $T$ be a nontrivial normal pro-$2'$ subgroup of $G$. By the Feit-Thompson theorem $G$ has a nontrivial Sylow 2-subgroup $U$. Set $U_1=C_U(T)$. Since $G$ is CN, we conclude that $U_1=U\cap F(G)$. By Lemma \ref{golova} $U/U_1$ is either cyclic or (generalized) quaternion. Note that finite groups with cyclic Sylow 2-subgroups are soluble \cite[Theorem 7.6.1]{go}. Taking into account that $G$ is not prosoluble we deduce that $U/U_1$ is (generalized) quaternion. 

Suppose that $U_1=1$, in which case $U$ is a finite (generalized) quaternion group. In view of the Feit-Thompson theorem, $G$ does not possess a normal 2-complement (otherwise $G$ would be prosoluble) and so by the  normal $p$-complement theorem of Frobenius $G$ has an element $x$ of odd order and a subgroup $U_0\leq U$ such that $x\in N_G(U_0)\setminus C_G(U_0)$. Let $y$ be the (unique) involution in $U$. It is straightforward that $\langle x,U_0\rangle\leq C_G(y)$. Therefore $C_G(y)$ is not pronilpotent, a contradiction.

Hence, without loss of generality we can assume that $U_1\not=1$ and so $TU_1\leq F(G)$. In particular, $T$ is a Cartesian product of its Sylow subgroups. Applying Lemma \ref{hap3} we conclude that $G/F(G)$ has no subgroups of Frobenius. Because of Lemma \ref{frobe} this yields a contradiction.
\end{proof}

\begin{theorem}\label{222} Let $G$ be a profinite CN-group which is non-prosoluble. Then $G/O_2(G)$ is almost simple. 
\end{theorem}
\begin{proof} In view of Theorem~\ref{odd}, $G$ is virtually pro-$2$. Let $M/O_2(G)$ be a minimal normal subgroup of $G/O_2(G)$. For an odd prime $p\in\pi(M)$, let $P$ be a $p$-Sylow subgroup of $M$. By the Frattini argument, $G=MN_G(P)$. Theorem \ref{odd} shows that $N_G(P)$ is prosoluble.
If $M/O_2(G)$ is abelian, then $G$ is prosoluble, contrary to the hypothesis. Thus, $G/O_2(G)$ does not have normal soluble subgroups.

Therefore $M/O_2(G)$ is a direct product of isomorphic non-abelian simple groups. Suppose that $M/O_2(G)$ is not simple. Then, for any  odd prime $p\in\pi(M)$ the $p$-Sylow subgroup of $M$ is not cyclic. In view of Lemma \ref{eleme} this is a contradiction. Hence $M/O_2(G)$ is simple. Taking into account that $G/O_2(G)$ does not have nontrivial normal soluble subgroups and putting this together with the fact that $G=MN_G(P)$, where $N_G(P)$ is prosoluble, we conclude that $M/O_2(G)$ is the unique minimal normal subgroup of $G/O_2(G)$. This means that $G/O_2(G)$ is almost simple.
\end{proof}

\subsection{Examples}

Let $G$ be an infinite profinite CN-group and $F$ the maximal normal pronilpotent subgroup of $G$. Theorem \ref{main} tells us that $G/F$ is finite.
\medskip

1. First, we show by examples that $G/F$ can be cyclic, or a direct product of a cyclic group of odd order and a (generalized) quaternion group.
Let $K$ be a either a finite cyclic group or a direct product of a cyclic group of odd order and a (generalized) quaternion group. Let $p$ be a prime which does not divide $|K|$ and let $V$ be a Cartesian product of infinitely many copies of the cyclic group of order $p$. Since $K$ has the structure of a Frobenius complement, $V$ admits an action of $K$ by automorphisms such that $C_V(x)$ is trivial for any $1\neq x\in K$. Thus, the semidirect product $VK$ has a natural structure of a profinite CN-group with $V$ being the maximal normal pronilpotent subgroup.
\medskip

2. Next, we produce an example where $G/F$ is a Frobenius group with cyclic kernel of odd order and cyclic complement. Our example is similar to so called finite 3-step groups (cf. \cite[p. 401]{go}). Let $D$ be a finite dihedral group of order 2$m$, where $m\geq3$ is odd. Remark that $D$ is a Frobenius group. Let $V$ be a Cartesian product of infinitely many copies of the group of order $2$. Note that $V$ admits an action of $D$ by automorphisms such that $C_V(x)$ is trivial for any nontrivial element of odd order $x\in D$. Thus, the semidirect product $VD$ has a natural structure of a profinite CN-group with $V$ being the maximal normal pronilpotent subgroup.
\medskip

3. Now we deal with an example in which $G/F$ is isomorphic to the group $SL(2,3)$. We will use results obtained in the construction of a similar example in \cite[Section 4]{ppz}. The reader therefore is referred to \cite[Section 4]{ppz} for details.

Let $S=SL(2,3)=\langle a,d\rangle$, where $a$ is of order 4 and $d$ of order 3. The $2$-Sylow subgroup in $S$ is $Q=\langle a,a^d\rangle$ and $a^2$ is central in $S$.

Let $\mathbb{Z}$ and $\mathbb{Z}_p$ stand for the rings of integers and $p$-adic integers, respectively. Let $V$ be the $4$-dimensional free $\mathbb{Z}$-module. Further, for a prime $p$ let $V_p$ be the $4$-dimensional free $\mathbb{Z}_p$-module.

The group $S$ can be embedded into $GL(V)$ in such a way that $C_V(x)=0$ for each nontrivial $x\in S$. It follows that for any prime $p$ the group $S$  embeds into $GL(V_p)$ in such a way that $C_{V_p}(x)=0$ for each nontrivial $x\in S$. To see this simply observe that 1 is not an eigenvalue for $x$. 

Let $p$ be an odd prime and set $U=V_p\oplus V_2$. We will view $S$ as a group of automorphisms of $U$ and $V_p$ and $V_2$ as $S$-invariant subgroups. Note that $C_{U}(x)=0$ for each nontrivial $x\in S$. Of course $a^2$ acts on $U$ taking each $u\in U$ to $-u$. Let $H$ be the natural semidirect product of $U$ by $S$. Clearly, $C_H(d)=\langle da^2\rangle$. 
Choose any nonzero element $v\in V_2$ and consider the subgroup $G=\langle va,d,V_p\rangle$. Let $W=V_2\cap G$ and $N=W\oplus V_p$. It is shown in \cite{ppz} that $\langle va,d\rangle$ is an infinite profinite group with abelian centralizers. Since the quotient $\langle va,d\rangle/W$ is isomorphic to $SL(2,3)$, we have $G/N\cong SL(2,3)$. It follows that $G$ is a profinite CN-group (actually the centralizers in $G$ are abelian) with the maximal normal pronilpotent subgroup $N$.
\medskip

4. Finally, we will show that $G/O_2(G)$ can be a non-abelian simple group. Let $S$ be a finite simple group isomorphic to either $PSL(2,2^m)$ for some $m\geq2$ or the Suzuki group $Sz(q)$. Note that $S$ is a finite CN-group. The group $S$ is a $2'$-semiregular group, that is, there exists a finite-dimensional $S$-module $M$ over a finite field of characteristic 2 such that $C_M(x)=0$ for each nontrivial $2'$-element $x\in S$ (see \cite{fleis} for details on finite $p'$-semiregular groups). Let $V$ be a Cartesian product of infinitely many copies of the group of order $2$. Note that $V$ admits an action of $S$ by automorphisms such that $C_V(x)$ is trivial for any nontrivial element of odd order $x\in S$. Thus, the semidirect product $VS$ has a natural structure of a profinite CN-group with $V$ being the maximal normal pronilpotent subgroup.

\section{Acknowledgments}
The author is grateful to CNPq and FAPDF for financial support.

\label{bibliography}

\end{document}